\documentclass[12pt,reqno]{amsart}

\usepackage{multirow}
\usepackage{hhline}
\usepackage{ulem}

\usepackage{mathtools}
\usepackage{times}
\usepackage[T1]{fontenc}
\usepackage{mathrsfs}
\usepackage{latexsym}
\usepackage[dvips]{graphics}
\usepackage[titletoc, title]{appendix}
\usepackage{amsmath,amsfonts,amsthm,amssymb,amscd}
\usepackage[dvipsnames]{xcolor}
\usepackage{hyperref}
\usepackage{amsmath}
\usepackage[shortlabels]{enumitem}
\usepackage{bm}
\usepackage{breakurl}

\usepackage{comment}
\newcommand{\bburl}[1]{\textcolor{blue}{\url{#1}}}

\newtheorem{thm}{Theorem}[section]

\newtheorem{cor}[thm]{Corollary}

\newtheorem{lem}[thm]{Lemma}
\newtheorem{prop}[thm]{Proposition}

\newtheorem{defi}[thm]{Definition}
\newtheorem{rek}[thm]{Remark}

\DeclareMathOperator{\supp}{supp}

\DeclareMathOperator{\sgn}{sgn}
\numberwithin{equation}{section}

\DeclareFontFamily{U}{mathx}{}
\DeclareFontShape{U}{mathx}{m}{n}{<-> mathx10}{}
\DeclareSymbolFont{mathx}{U}{mathx}{m}{n}
\DeclareMathAccent{\widehat}{0}{mathx}{"70}
\DeclareMathAccent{\widecheck}{0}{mathx}{"71}

\def\BB{{\mathcal{B}}}
\def\YY{\mathbb{Y}}
\def\ZZ{\mathbb{Z}}
\def\XX{\mathbb{X}}
\def\NN{\mathbb{N}}
\def\CC{\mathbf{C}}
\def\DD{\mathcal{D}}
\def\E{\mathcal{E}}
\def\II{\mathcal{I}}
\def\PG{{\mathbf{P}^{con}_{g,\tau}}}
\def\CG{{\mathbf{C}^{con}_{g,\tau}}}
\def\KK{{\mathbf{K}}}
\def\M{{\mathbf{M}}}
\def\FF{{\mathbb{F}}}
\def\RR{{\mathbb{R}}}
\def\CCC{{\mathbb{C}}}

\def\ee{{\bm e}}
\def\ff{{\bm f}}

\def\1{{\mathbf 1}}

\begin{document}

\title{On Consecutive Greedy and other greedy-like type of bases. }

\author{Miguel Berasategui}

\email{\textcolor{blue}{\href{mailto: mberasategui@dm.uba.ar}{mberasategui@dm.uba.ar}}}
\address{IMAS - UBA - CONICET - Pab I, Facultad de Ciencias Exactas
y Naturales, Universidad de Buenos Aires, (1428), Buenos Aires, Argentina}

\author{Pablo M. Bern\'a}

\email{\textcolor{blue}{\href{mailto:pablo.berna@cunef.edu}{pablo.berna@cunef.edu}}}
\address{Departmento de Métodos Cuantitativos, CUNEF Universidad, 28040 Madrid, Spain}

\author{H\`ung Vi\d{\^e}t Chu}

\email{\textcolor{blue}{\href{mailto:hungchu2@illinois.edu}{hungchu2@illinois.edu}}}
\address{Department of Mathematics\\ University of Illinois at Urbana-Champaign, Urbana, IL 61820, USA}

\begin{abstract} 
We continue our study of the Thresholding Greedy Algorithm when we restrict the vectors involved in our approximations so that they either are supported on intervals of $\NN$ or have constant coefficients. We introduce and characterize what we call consecutive greedy bases and provide new characterizations of almost greedy and squeeze symmetric Schauder bases. Moreover, we investigate some cases involving greedy-like properties with constant $1$ and study the related notion of Property (A, $\tau$).
\end{abstract}



\maketitle

\section{Introduction}
In the field of non-linear approximation, one of the objectives is to find natural algorithms to approximate elements of a given space. Formally, given a (Markushevich) basis $(\ee_n)_n$ of a certain space $\XX$, we associate with each element $x$ in $\mathbb{X}$ the formal series $\sum_n a_n \ee_n$, where $a_n$'s are scalars, called the coefficients of $x$. Examples of such representations include the Taylor expansions or the Fourier series. We then find approximations of $f$ in terms of finite sums 
$\sum_{n\in A(x)}b_n \ee_n,$
for a suitable set $A(x)$ (depending on $x$) and scalars $b_n$ (possibly different from $a_n$). 

For the past twenty years, an algorithm that has attracted much attention is the Thresholding Greedy Algorithm $(G_m)_m$ (TGA), which selects the largest coefficients of $x$ in modulus; that is, if $(a_1,a_2,\ldots)$ is the sequence of the coefficients of  $x\in\XX$, we rearrange the coefficients by their magnitude so that  $\vert a_{\pi(1)}\vert\geqslant \vert a_{\pi(2)}\vert\geqslant\cdots$ and let $G_m(x) = \sum_{n=1}^m a_{\pi(n)}\ee_n$.

We would like to know when the TGA produces the best possible approximation, i.e., when
\begin{eqnarray}\label{int1}
	\Vert x-G_m(x) \Vert\ \approx\ \sigma_m(x), 
\end{eqnarray}
for each $m\in \NN$, where $\sigma_m(x):=\inf_y \Vert x-y\Vert$, and $y$ is any linear combination of basis vectors with $|\supp(y)|\leqslant m$. In \cite{KT1}, the authors introduce the notion of greedy bases as those  satisfying \eqref{int1}. Since then, many authors have explored the TGA and introduced different greedy-type bases, such as almost greedy bases, partially greedy bases, and semi-greedy bases.
Out of these, partially greedy bases are particularly important because they compare the non-linear greedy approximation and the standard linear one. In particular, a basis is partially greedy when 
$$\Vert x-G_m(x)\Vert \ \lesssim\ \Vert x-S_m(x)\Vert,$$
where $S_m(x)$ is the sequence of the linear algorithm formed by the partial sums. Inspired by the contrast between linearity and non-linearity and also following up on the research started in \cite{BBC}, we propose in this paper the following modification of the error $\sigma_m(x)$: we consider the error $\sigma_m^{con}(x):= \inf_y \Vert x-y\Vert$, where $y$ is  generated by basis vectors and supported on an interval of consecutive natural numbers of length $m$. Surprisingly, if we consider bases satisfying $$\Vert x-G_m(x)\Vert\ \lesssim\ \sigma_m^{con}(x),$$
we do not recover the best possible approximation but we achieve convergence based on the concept of almost greedy bases, for which  the TGA produces the best approximation using projections. In the context of this research, we also prove some new equivalences involving symmetry for largest coefficients - also known as Property (A) - and squeeze-symmetry,  two concepts of great interest in the theory of greedy approximation and which have been studied in several papers in the field.

\section{Background and notation}\label{sectionintroduction}

Let $\mathbb{X}$ be a separable, infinite dimensional $p$-Banach space $(0<p\leqslant 1)$ over the field $\mathbb{F} = \{\mathbb{R}, \mathbb{C}\}$. Let $\mathbb{X}^*$ be the dual space of $\mathbb{X}$. By a \textbf{basis}, we mean a sequence $\mathcal{B}= (\ee_n)_{n\in \NN}\subset \mathbb{X}$ such that  
\begin{enumerate}
    \item $\mathbb{X} = \overline{[\ee_n:n\in \NN]}$, where $[\ee_n:n\in \NN]$ denotes the span of $(\ee_n)_{n\in \NN}$;
    \item there is a unique sequence $\BB^*=(\ee_n^*)_{n\in \NN}\subset \mathbb{X}^*$ such that $\ee_j^*(\ee_k) = \delta_{j, k}$ for all $j, k\in\mathbb{N}$;
    \item there exist $c_1, c_2 > 0$ such that 
    $$0 \ <\ c_1 :=\ \inf_n\{\|\ee_n\|, \|\ee_n^*\|\}\ \leqslant\ \sup_n\{\|\ee_n\|, \|\ee_n^*\|\} \ =:\ c_2 \ <\ \infty.$$
\end{enumerate}
If $\mathcal{B}$ also satisfies 
\begin{enumerate}
\item[(4)] $\mathbb{X}^*\ =\ \overline{[\ee_n^*: n\in \mathbb{N}]}^{w^*},$
\end{enumerate}
then $\mathcal{B}$ is a \textbf{Markushevich basis}. Additionally, if the partial sum operators $S_m(x) = \sum_{n=1}^m \ee_n^*(x)\ee_n$ for $m\in \mathbb{N}$ are uniformly bounded, i.e., there exists $\mathbf C > 0$ such that 
\begin{enumerate}
\item[(5)] $\|S_m(x)\|\ \leqslant\ \mathbf C\|x\|, \forall x\in \mathbb{X}, \forall m\in \mathbb{N}$,
\end{enumerate}
then $\BB$ is a \textbf{Schauder basis} (or a $\CC$-Schauder basis, in particular). The minimum $\CC$ for which the above holds is called the \textit{basis constant}, denoted by $\KK_b$. 

If $\BB$ satisfies the stronger condition that the inequality 
$$
\left\|P_A(x)\ :=\ \sum_{n\in A}\ee_n^*(x)\ee_n \right\|\ \leqslant\ \CC\|x\|
$$
holds uniformly for some fixed $\CC>0$ across all $x\in \XX$ and finite $A\subset \NN$, we say that $\BB$ is $\CC$-suppression unconditional. For an unconditional basis, let $\mathbf K_{su} > 0$ be the smallest such that
$$\|x-P_A(x)\|\ \leqslant\ \mathbf K_{su}\|x\|, \forall x\in \mathbb{X}, \forall \mbox{ finite }A\subset \mathbb{N},$$
and $\BB$ is said to be $\mathbf K_{su}$-suppression unconditional.  We should use ``suppression'' both times, or neither. 
In the sequel, $\BB$ will always denote a basis of a $p$-Banach space $\XX$.

For each $x\in \mathbb{X}$, we want an efficient algorithm to extract finite sums from the formal series $\sum_{n=1}^\infty \ee_{n}^*(x)\ee_n$ and see how well these sums approximate our vector $x$. In 1999, Konyagin and Temlyakov \cite{KT1} introduced the Thresholding Greedy Algorithm (TGA) which chooses the largest coefficients of $x$ in modulus, i.e., largest $|e_n^*(x)|$ to include in the sums. In this paper, we will work with a weaker, more general version of the algorithm (WTGA) first studied by Temlyakov \cite{T98}. The WTGA allows more flexibility in forming sums by choosing largest coefficients up to a constant. In particular, given $0<\tau\leqslant 1$ and $x\in \XX$, we say that $A\subset \mathbb{N}$ is a $\tau$-\textbf{greedy set }of $x$ of order $m$ if $|A| = m$ and 
$$\min_{n\in A}|\ee_n^*(x)|\ \geqslant\ \tau \max_{n\notin A}|\ee_n^*(x)|.$$
In the case $\tau=1$, the WTGA coincides with the TGA, and in this case, we call $A$ a greedy set of $x$ of order $m$; the set of all $\tau$-greedy sets of $x$ of order $m$ is denoted by $G(x,m,\tau)$, whereas $G(x)$ denotes the set of all finite greedy sets of $x$. The corresponding sum is $\sum_{n\in A}\ee_n^*(x)\ee_n$. It is worth noting that for $\tau < 1$, the WTGA is more flexible than the TGA; however, the flexibility has a minimal effect on the approximation efficiency (see \cite[Section 1.5]{T2008}.) 

In \cite{KT1}, the authors introduced greedy and quasi-greedy bases and provided a key characterization of the former as being unconditional and democratic. Later, it was proven in \cite{AABW2021} that this characterization also holds for $p$-Banach spaces. 

\begin{defi}\rm A basis $\BB$ is quasi-greedy with constant $\CC>0$ (or $\CC$-quasi-greedy) if 
$$
\|P_{A}(x)\|\ \leqslant\ \CC\|x\|, \forall x\in \XX, \forall m\in \NN, \forall A\in G(x,m,1).
$$
\end{defi}

\begin{defi}\rm A basis $\BB$ is greedy with constant $\CC>0$ (or $\CC$-greedy) if 
\begin{equation}
\|x-P_{A}(x)\|\ \leqslant\ \CC\sigma_{m}(x), \forall x\in \XX, \forall m\in \NN, \forall A\in G(x,m,1), \label{greedy}
\end{equation}
where $\sigma_m(x)$ is the best  $m$-term approximation error defined as 
$$\sigma_m(x)\ =\ \inf_{\substack{y\in \XX\\|\supp(y)|\leqslant m}}\|x-y\|,$$
and $\supp(y):=\{n\in \NN: \ee_n^*(y)\neq0\}$ is the \textit{support} of $y$ (with respect to $\BB$). 
\end{defi}
The above definitions conceptualize efficiency requirements. Specifically, for a quasi-greedy basis, the TGA produces sums that converge to the to-be-approximated vector (\cite[Theorem 4.1]{AABW2021}, \cite{W2000}), which is the minimal requirement for a sensible algorithm; for a greedy basis, the TGA produces essentially the best approximation up to a constant.

For the next definitions, we use the following notation: 
\begin{itemize}
\item $\NN^{<\infty}$ is the set of all finite subsets of $\NN$, and $\NN^{(m)}$ is set of subsets of $\NN$ with cardinality $m$. 
\item $\E:=\{(a_n)_{n\in\NN}\in \FF^{\NN}: |a_n|=1\;\forall n\in\NN\}$ is the set of \textit{signs}.
\item Given $A\in \NN^{<\infty}$ and $\varepsilon\in \E$, we write 
$$
\1_{A}\ :=\ \sum_{n\in A}\ee_n\mbox{ and }\1_{\varepsilon,  A}\ :=\ \sum_{n\in A}\varepsilon_n \ee_n
$$
\item Given $x\in \XX$, we write 
$$
\varepsilon(x)\ :=\ \left(\sgn\left(\ee_n(x)\right)\right)_{n\in \NN}\ \in\ \E,
$$
where $\sgn(0)=1$. 
\end{itemize}

\begin{defi}\rm Let $\CC>0$. A basis $\BB$ is $\CC$-superdemocratic if 
$$\|\1_{\varepsilon,  A}\|\ \leqslant\ \CC \|\1_{\delta, B}\|, \forall A,B\in \NN^{<\infty} \mbox{ with } |A|\leqslant |B|, \forall \varepsilon, \delta \in \E,$$
and it is $\CC$-democratic when the above holds with $\varepsilon_n=\delta_n=1$ for all $n$. 
\end{defi}

\begin{thm}\cite[Corollary 7.3]{AABW2021}, \cite{KT1}\label{KT1}
A basis $\mathcal{B}$ is greedy if and only if
it is unconditional and (super)democratic.
\end{thm}
Between quasi-greedy and greedy bases lie almost greedy bases, introduced by Dilworth, Kalton, Kutzarova, and Temlyakov in \cite{DKKT2003}. The idea is to compare the approximation efficiency of sums from TGA against all possibble projections instead of all possible linear combinations as in the definition of greedy bases. As a result, the almost greedy property is weaker and is much more common than the greedy property (\cite[Remark 7.7]{DKK2003}.) Formally,

\begin{defi}\rm A basis $\BB$ is almost greedy with constant $\CC>0$ (or $\CC$-almost greedy) if 
$$
\|x-P_{A}(x)\|\ \leqslant\ \CC\widetilde{\sigma}_{m}(x), \forall x\in \XX, \forall m\in \NN, \forall A\in G(x,m,1), 
$$
where $\widetilde{\sigma}_m(x)$ is the best  $m$-term approximation error by projections, defined as 
\begin{equation}
\widetilde{\sigma}_m(x)\ =\ \inf_{A\in \NN^{(m)}}\|x-P_A(x)\|.\label{almostgreedy}
\end{equation}
\end{defi}
Almost greedy bases have a characterization close to that of greedy bases by substituting quasi-greediness for unconditionality - a characterization proved in \cite{DKKT2003} for Banach spaces and  in \cite{AABW2021} for $p$-Banach spaces. 
\begin{thm}\label{theorem: DKKT3.3}\cite[Proposition 6.6]{AABW2021}, \cite[Theorem 3.3]{DKKT2003}
A basis $\mathcal{B}$ is almost greedy if and only if
it is quasi-greedy and (super)democratic.
\end{thm}
Since their introduction, the TGA and the aforementioned bases have been widely studied in the literature, alongside other greedy-like types of bases  (see \cite{AA2016, AA2017, AABW2021, DKK2003, DKKT2003, T2008, W2000}.) Further background about this topic can be found in \cite{AABW2021, T2008, T2011}.

This paper continues our study of the performance of the TGA against approximations restricted to vectors supported on intervals and vectors with constant coefficients: \cite[Theorem 1.7]{BBC} states that if we restrict the approximations in \eqref{almostgreedy} to intervals of $\NN$, we still obtain almost greedy bases. This is an unexpected result since the consecutive restriction depends on the basis order, while the almost greedy property is order-free. Nevertheless, the restriction still characterizes the almost greedy property. In the same paper, motivated by results in \cite{BB2017, BC, DK2019}, we characterized almost greedy bases using approximations in terms of vectors with constant coefficients and limited to intervals of $\NN$ (\cite[Proposition 1.14]{BBC}). 

The goal of this paper is to study similar restrictions of the approximations appeared in the definition of greedy bases. Unlike the case of the almost greedy property, these restrictions do not produce an equivalence of the greedy property. Besides, we study other properties that appear naturally in this context. In particular, Section~\ref{sectionconsecutivegreedy} defines consecutive greedy bases and characterizes them in the classical sense.  Section~\ref{otherchars} gives yet other characterizations of consecutive greedy bases using the so-called Property (A, $\tau$) (introduced in \cite{C3} as a generalization of Property (A), which was first appeared in \cite{AW2006}). There we also characterize consecutive greedy bases with constant $1$. In Section \ref{Atau}, we study Property (A, $\tau$) more in depth. First, we optimize the constant in \cite[Corollary 5.5]{C3}. Second, we answer \cite[Problem 7.3]{C3} in the affirmative  by showing that once we have Property (A, $\tau$) for some $\tau\in (0,1]$, then we have Property (A, $\tau'$) for all $\tau'\in (0,1]$. The proof leverages 
a recent result in \cite{AAB2022}. In Section~\ref{sectionsqueezesymmetric}, we study greedy-like algorithms in which we restrict the approximations on both sides of \eqref{greedy} to vectors of constant coefficients with respect to a basis $\BB$. Finally, in the Appendix, we consider consecutive-like approximations but on the side of the greedy (or greedy-like) sets.

\section{Consecutive greedy bases}\label{sectionconsecutivegreedy}

In this section, we introduce and characterize consecutive greedy bases. We will use the following notation. 
\begin{itemize}
\item  $\II$ denotes the set of all intervals of $\NN$. 
\item  For each $m\in \NN_0$, $\II^{(m)}$ denotes the set of all intervals of $\NN$ of length $m$. 
\item $\tau$ is any real number in $(0,1]$. 
\item For $x\in \mathbb{X}$, let $G^\tau_m(x) = P_{A}(x)$ for some $A\in G(x, m, \tau)$.
\end{itemize}
 We begin with a definition that is inspired by consecutive almost greedy bases from \cite{BBC} but considers the WTGA and approximations by arbitrary vectors rather than projections.

 \begin{defi}\label{definitionconsecutivegreedy}\normalfont
A basis $(\ee_n)_n$ is said to be ($\mathbf C$, $\tau$)-consecutive greedy for some constant $\mathbf C\geqslant 1$ if
$$\|x-G^{\tau}_m(x)\|\ \leqslant\ \mathbf C\sigma^{con}_m(x),\forall x\in\mathbb{X}, \forall m\in\mathbb{N},\forall G_m^{\tau}(x),$$
where$$\sigma^{con}_m \ =\  \inf\left\lbrace \left\|x - \sum_{n\in I}a_n\ee_n\right\|\,:\, a_n\in\mathbb F, I\in \II^{(m)}\right\rbrace.$$
In this case, the least constant is denoted by $\mathbf C^{con}_{g, \tau}$. When $\tau = 1$, we say that $(\ee_n)_n$ is consecutive greedy. 
\end{defi}

In \cite[Corollary 1.8]{BB2017}, the authors proved that if there is a constant $\CC>0$ such that \eqref{greedy} holds when $\sigma_m(x)$ is replaced by 
$$\DD_m(x)\ =\ \inf_{\substack{A\in \NN^{(m)}\\\lambda\in \FF}}\|x-\lambda\1_{A}\|.$$
As we shall see, there is a corresponding characterization that holds for consecutive greedy bases. First, we define a property that is formally weaker than consecutive greediness. 
\begin{defi}\label{definitionpoly}\normalfont
		For some constant $\mathbf C\geqslant 1$, a basis $(\ee_n)_n$ is said to be ($\mathbf C$, $\tau$)-consecutive greedy for polynomials of constant coefficients if
		$$\|x-G^{\tau}_m(x)\|\ \leqslant\ \mathbf C\mathcal D^{con}_m(x),\forall x\in\mathbb{X}, \forall m\in\mathbb{N}, \forall G_m^{\tau}(x),$$
		where $$\mathcal D^{con}_m(x) \ =\  \inf_{\substack{I\in \II^{(m)}\\\lambda\in\mathbb{F}}}\left\|x - \lambda \1_{I}\right\|.$$
		In this case, the least constant is denoted by $\mathbf P^{con}_{g, \tau}$. When $\tau = 1$, we say that $(\ee_n)_n$ is consecutive greedy for polynomials of constant coefficients (CGPCC). 
\end{defi}
To prove the equivalence between consecutive greedy and CGPCC bases, as well as their characterization in terms of well-known properties, we will use the following equivalent variant of quasi-greediness.

\begin{defi}\label{definitiontausuppressionquasigreedy}\normalfont
A basis $(\ee_n)_n$ is said to be $\mathbf C$-suppression $\tau$-quasi-greedy for some constant $\mathbf C\geqslant 1$ if
$$\|x-G^{\tau}_m(x)\|\ \leqslant\ \mathbf C\|x\|,\forall x\in\mathbb{X}, \forall m\in\mathbb{N},\forall G_m^{\tau}(x).$$
In this case, the least constant is denoted by $\mathbf C_{\ell, \tau}$. When $\tau = 1$, we simply say that $(\ee_n)_n$ is suppression quasi-greedy. 
\end{defi}

\begin{rek}\label{r1}\normalfont
By a result of Konyagin and Temlyakov \cite{KT2}, if $(\ee_n)_n$ is $\mathbf C$-suppression $\tau$-quasi-greedy for some $\tau$ and $\mathbf C$, then it is ($\mathbf D$, $\tau'$)-suppression $\tau'$-quasi-greedy for all $\tau'\in (0, 1]$ and for some $\mathbf D$ possibly dependent on $\tau'$. 
\end{rek}

We also need the notion of \textit{truncation quasi-greedy bases}, which appear naturally in the study of the TGA \cite{AABBL2021,AABW2021, DKKT2003}. 

\begin{defi}\normalfont
A basis $(\ee_n)_n$ is said to be $\CC$-truncation quasi-greedy for some constant $\mathbf C\geqslant 1$ if
$$\min_{n\in A}|\ee_n^*(x)|\|1_{\varepsilon(x), A}\|\ \leqslant\ \CC \|x\|, \forall x\in \XX, \forall A\in G(x).$$
The least constant above is denoted by $\mathbf C_{tq}$. 
\end{defi}
\begin{rek}\label{r3}\rm It was proven in \cite[Lemma 2.2]{DKKT2003} and \cite[Theorem 4.13]{AABW2021} that quasi-greedy bases are truncation quasi-greedy. 
\end{rek}

We now state our first main result.

\begin{thm}\label{m1}
Let $\BB$ be a basis. The following hold
	\begin{enumerate}
		\item [i)] If $\BB$ is $(\CG, \tau)$-consecutive greedy, then it is $(\CG)^2 $-superdemocratic and $(\CG, \tau)$-CGPCC.  
		\item [ii)] If $\BB$ is $(\PG, \tau)$-CGPCC, then it is $\PG$-suppression $\tau$-quasi-greedy,  $(\PG)^2$-democratic, and $\PG$-Schauder.
		\item [iii)] If $\BB$ is $\mathbf C_{\ell, \tau}$-suppression $\tau$-quasi-greedy, $\KK_b$-Schauder, $\CC_{tq}$-truncation quasi-greedy, and $\Delta_{sd}$-superdemocratic,  then $\BB$ is ($\mathbf C^{con}_{g, \tau}$, $\tau$)-consecutive greedy with 
		$$\CG\ \leqslant\ \left(\left(1+2\KK_b^p\right)\left(2+\CC_{\ell,\tau}^p\right)+\left(\tau^{-1}A_p\Delta_{sd}\CC_{tq}\right)^{p}\right)^{\frac{1}{p}}.$$
		In the case $p=1$, we also have
$$
\CG\ \leqslant\ \left(1+2\KK_b\right)\left(2+\CC_{\ell,\tau}\right)+2\tau^{-1}\Delta_{sd}\CC_{\ell,1}.
$$		
	\end{enumerate}
\end{thm}

\begin{cor}
Let $(\ee_n)_n$ be a Schauder basis. Then $(\ee_n)_n$ is almost greedy if and only if it is consecutive greedy. 
\end{cor}

\begin{proof}
If $(\ee_n)_n$ is almost greedy, then Theorem \ref{theorem: DKKT3.3} says that $(\ee_n)_n$ is quasi-greedy and superdemocratic. Now use Remark \ref{r3} and Theorem \ref{m1} item iii). Conversely, if $(\ee_n)_n$ is consecutive greedy, then by Theorem \ref{m1} items i) and ii), $(\ee_n)_n$ is superdemocratic and quasi-greedy. Invoking Theorem \ref{theorem: DKKT3.3} completes the proof. 
\end{proof}

\begin{rek}\rm
Note that any conditional almost greedy Schauder basis is consecutive greedy but not greedy; many examples of such bases can be found in the literature (see \cite{AABBL2021, AK, DKK2003, KT1, T2008, T2011, W2000}, to name a few.)
\end{rek}
Before proving Theorem~\ref{m1}, we need an auxiliary lemma, in which we use the following definitions from \cite{AABW2021}.
\begin{itemize}
\item For $0<p\leqslant 1$, let $A_p:=(2^p-1)^{-\frac{1}{p}}$ and $B_p=(2\kappa)^{\frac{1}{p}}A_p$, where $\kappa=1$ if $\FF=\RR$ and $\kappa=2$ if $\FF=\mathbb{C}$. 
\item For $u>0$, let 
$$
\eta_p(u)\ :=\ \min_{0<t<1}(1-t^p)^{-\frac{1}{p}}(1-(1+A_p^{-1}u^{-1}t)^{-p})^{-\frac{1}{p}}. 
$$
\end{itemize}

\begin{lem}\label{lemmatqg} If $(\ee_n)_{n}$ is $\CC$-suppression quasi-greedy, it is $\CC_p$-truncation quasi-greedy with 
\begin{equation*}
\CC_p\ :=\ \begin{cases}
2\CC & \text{ if } p=1;\\
\CC^2\eta_p(\CC) & \text{ if } 0<p<1.
\end{cases}
\end{equation*}
\end{lem} 
\begin{proof}
The case $0<p<1$ follows from \cite[Theorem 4.13 (i)]{AABW2021}.
The case $p = 1$ follows from \cite[Lemma 2.3]{BBG2017}. Below we offer an alternate proof, which involves only a change of basis instead of integration, for the case $p = 1$.
If $\FF=\RR$ the result follows from the proof of \cite[Lemma 2.2]{DKKT2003}. Suppose that $\FF=\CCC$. Fix $x\in \XX$ with finite support, $m\in \NN$, and $A\in G(x,m,1)$. 

Choose signs $(\varepsilon_n)_{n\in \NN}$ so that if $\ff_n^*=\varepsilon_n^{-1}\ee_n^*$, then $\ff_n^*(x)\geqslant 0$ for all $n\in \NN$. Let $\ff_n=\varepsilon_n \ee_n$ for $n\in \NN$, and let $\YY$ be the Banach space over $\RR$ defined as the completion of
$$
\ZZ\ :=\ \left\lbrace \sum_{n=1}^{k} a_n \ff_n: k\in \NN, a_n\in \RR \right\rbrace
$$
with the norm given by the restriction to $\ZZ$ of the norm of $\XX$. For each $z\in \ZZ$, let $\delta(z):=\left(\sgn\left(\ff_n^*(z)\right)\right)_{n}$. 

Clearly, we may assume that $A\subset \supp(x)$. Identify each $\ff_n^*$ with its counterpart in $\ZZ^*$. As $x\in \ZZ$  and $(\ff_n)_{n}$ is also $\CC$-suppression quasi-greedy, applying the result for $\FF=\RR$, we obtain
\begin{align*}
\min_{n\in A}|e_n^*(x)|\|\1_{\varepsilon(x), A}\|_{\XX}&\ =\ \min_{n\in A}|\ee_n^*(x)|\left\| \sum_{n\in A}\frac{\ee_n^*(x)}{|\ee_n^*(x)|}  \ee_n\right\|_{\XX}\\
&\ =\ \min_{n\in A}|\ff_n^*(x)|\left\| \sum_{n\in A}\frac{\delta_n(x)}{\varepsilon_n} \ff_n\right\|_{\ZZ}\ =\ \min_{n\in A}|\ff_n^*(x)|\|\1_{\delta', A}\|_{\ZZ}\\
&\ \leqslant\  2\CC\|x\|_{\ZZ}\ =\ 2\CC\|x\|_{\XX},
\end{align*}
where $\delta'_n = \delta_n(x)/\varepsilon_n \in \RR$ for all $n\in A$.
The result for infinitely supported $x$ follows from a standard density argument.
\end{proof}

\begin{proof}[Proof of Theorem \ref{m1}]
i) We only need to prove the superdemocracy property: Let $I_1\in \mathcal{I}$ be an interval containing $A$ and choose an interval $I_2>I_1\cup B$ with $|I_2|=|B|$. We have 
\begin{align*}
\|\1_{\varepsilon,  A}\|&\ =\ \|(\1_{\varepsilon,  A}+\1_{I_2} + \1_{I_1\backslash A}) -(\1_{I_2} + \1_{I_1\backslash A})  \|\\
&\ \leqslant\ \CG\|(\1_{\varepsilon,  A}+\1_{I_2}+\1_{I_1\backslash A}) - (\1_{\varepsilon,  A}+\1_{I_1\backslash A})\|\\
&\ =\ \CG\|1_{I_2}\|, 
\end{align*}
and 
$$
\|\1_{I_2}\|\ =\ \|(\1_{I_2}+\1_{\delta, B})-\1_{\delta, B}  \|\ \leqslant\ \CG\|\1_{\delta, B}\|. 
$$
Now the result follows by combining the above estimates. \\
ii) Since $\sigma^{con}_m(x)\leqslant\|x\|$ for all $x\in\mathbb{X}$, $(\ee_n)_n$ is $\PG$-suppression $\tau$-quasi-greedy. The proof of democracy is the same as that of superdemocracy in i), but removing the signs. To prove the Schauder condition, pick $x\in \mathbb X\setminus \{0\}$ with finite support and $m\in \NN$. Let $I:=\{1,\dots,m\}$ and $n:=\max(\supp(x))$. If $n\leqslant m$, then $P_{I}(x)=x$ and there is nothing to prove. Otherwise, let $I_2:=\{m+1,\dots, n\}$ and choose $\alpha>0$ sufficiently large so that if $y=x+\alpha \1_{I_2}$, then $I_2\in G(y,|I_2|,\tau)$. We have
$$
\|P_I(x)\|\ =\ \|y-P_{I_2}(y)\|\ \leqslant\ \PG\|y - \alpha \1_{I_2}\|\ =\ \PG\|x\|.
$$
iii) Let $x\in\mathbb{X}$, $m\in\mathbb{N}$, $\Lambda\in G(x, m, \tau)$, and $I\in\mathcal{I}^{(m)}$. Given scalars $(a_n)_{n\in I}$, set $y:=x- \sum_{n\in I}a_ne_n$. We have
\begin{equation}
\|x-P_{\Lambda}(x)\|^p\ \leqslant\ \|x-P_{I}(x)\|^p + \|P_{I\backslash \Lambda}(x)\|^p + \|P_{\Lambda\backslash I}(x)\|^p.\label{e1}
\end{equation}
On the one hand, since $(\ee_n)_n$ is Schauder, 
\begin{equation}\label{e2}\|x - P_{I}(x)\|^p\ =\ \|y - P_{I}(y)\|^p\ \leqslant\ (1+2\mathbf K_b^p )\left\|y\right\|^p.\end{equation}
On the other hand, since $\Lambda\backslash I$ is a $\tau$-weak greedy set of $x-P_I(x)$, we know that 
\begin{equation}\label{e3}\|P_{\Lambda\backslash I}(x)\|^p\ =\ \|P_{\Lambda\backslash I}(x-P_I(x))\|^p\ \leqslant\ (\mathbf C_{\ell, \tau}^p + 1)\|x-P_I(x)\|^p. 
\end{equation}
Finally, let $A= \{n: |\ee_n^*(y)|\geqslant \min_{n\in\Lambda\setminus I}|\ee_n^*(y)|\}$. We have
\begin{align}\label{e50}
    \|P_{I\backslash \Lambda}(x)\|&\ \leqslant\ A_p \max_{n\in I\backslash \Lambda}|\ee_n^*(x)|\sup_{\delta}\|\1_{\delta, (I\backslash \Lambda)}\|\ \leqslant\ \ A_p \tau^{-1}\min_{n\in \Lambda \backslash I}|e_n^*(y)|\sup_{\delta}\|\1_{\delta, (I\backslash \Lambda)}\|\nonumber \\
    &\ \leqslant\ \tau^{-1}A_p\Delta_{sd}\min_{n\in \Lambda \backslash I}|\ee_n^*(y)|\|\1_{\varepsilon(y), A} \|=\tau^{-1}A_p\Delta_{sd}\min_{n\in A}|\ee_n^*(y)|\|\1_{\varepsilon(y), A} \|\nonumber \\
    &\  \leqslant\ \tau^{-1}A_p \Delta_{sd}\CC_{tq}\|y\|\text{ by Lemma }\ref{lemmatqg}.
\end{align}
From \eqref{e1}, \eqref{e2}, \eqref{e3}, and \eqref{e50}, we obtain
$$
\|x-P_{\Lambda}(x)\|\ \leqslant\ \left(\left(1+2\KK_b^p\right)\left(2+\CC_{\ell,\tau}^p\right)+\left(\tau^{-1}A_p\Delta_{sd}\CC_{tq}\right)^{p}\right)^{\frac{1}{p}}\|y\|. 
$$
Hence, if $p=1$, by Lemma~\ref{lemmatqg}, we obtain 
\begin{equation*}
\|x-P_{\Lambda}(x)\|\ \leqslant\ \left(\left(1+2\KK_b\right)\left(2+\CC_{\ell,\tau}\right)+2\tau^{-1}\Delta_{sd}\CC_{\ell,1}\right)\|y\|.
\end{equation*} 
Since $I\in\mathcal{I}$ and $(a_n)_{n\in I}$ are arbitrary, we are done. 
\end{proof}


\begin{rek}\label{r2}\normalfont
Theorem \ref{m1} and Remark \ref{r1} imply that if $\BB$ is ($\mathbf C$, $\tau$)-consecutive greedy for some $\tau\in (0,1]$ and some $\mathbf C\geqslant 1$, then $(\ee_n)_n$ is ($\mathbf D$, $\tau'$)-consecutive greedy for all $\tau'\in (0,1]$ and some $\mathbf D$ possibly dependent on $\tau'$.
\end{rek}

\section{Characterizations of consecutive greedy bases involving Property (A, $\tau$) and $1$-consecutive greedy bases}\label{otherchars}

In addition to the classical charaterizations of Theorems~\ref{KT1} and~\ref{theorem: DKKT3.3}, greedy and almost greedy bases can be characterized as those that have Property (A) and are unconditional or quasi-greedy, respectively (see \cite[Theorem 2]{DKOSZ} and \cite[Theorems 4.1 and 4.3]{BDKOW2019}). In a similar manner, we will use an extension of Property (A) introduced in \cite{C3} to give another characterization of consecutive greedy bases. For a collection of sets $(A_i)_{i\in I}$, write $\sqcup_{i\in I} A_i$ to mean that the $A_i$'s are pairwise disjoint.  

\begin{defi}\label{defatau}\normalfont
A basis $(\ee_n)_n$ is said to have Property (A, $\tau$) if there exists a constant $\mathbf C = \mathbf C(\tau)\geqslant 1$ such that
$$\|\tau x+ \1_{\varepsilon,  A}\|\ \leqslant\ \mathbf C\|x+ \1_{\delta, B}\|,$$
for all $x\in\mathbb{X}$ with $\|x\|_\infty\leqslant 1/\tau$, for all finite sets $A, B\subset\mathbb{N}$ with $|A|\leqslant |B|$ and $\supp(x)\sqcup A\sqcup B$, and for all signs $\varepsilon, \delta$. When $\tau = 1$, we say that $(\ee_n)_n$ has Property (A). Furthermore, if the constant $\mathbf C$ is uniform across all $\tau\in (0,1]$, we say that $(\ee_n)_n$ has the uniform Property (A). 
\end{defi}

We state a result proven for Banach space, but the result can be easily carried over to $p$-Banach spaces.

\begin{prop}\label{ps}[Restatement of Proposition 5.4 in \cite{C3} for $p$-Banach spaces]
A quasi-greedy basis that has Property (A, $\tau$) for some $\tau\in (0,1]$ has the uniform Property (A).
\end{prop}

\begin{cor}\label{c2} Let $(\ee_n)_n$ be a basis of a $p$-Banach space. 
The following are equivalent
\begin{enumerate}
    \item [i)] $(\ee_n)_n$ is consecutive greedy.
    \item [ii)] $(\ee_n)_n$ is quasi-greedy, Schauder, and superdemocratic.
     \item [iii)] $(\ee_n)_n$ is both quasi-greedy and Schauder and has the uniform Property (A).
    \item [iv)] $(\ee_n)_n$ is both Schauder and quasi-greedy and has Property (A, $\tau$) for (some) all $\tau\in (0,1]$.
\end{enumerate}
\end{cor}

\begin{proof}

That i) $\Longleftrightarrow$ ii) is due to Theorem \ref{m1} and iii) $\Longrightarrow$ iv) is immediate. 
Moreover, Proposition \ref{ps} gives that iv) $\Longrightarrow$ iii).

To see that ii) $\Longrightarrow$ iii), we first use Theorem~\ref{theorem: DKKT3.3} to claim that a basis satisfying ii) is almost greedy. By \cite[Theorem 2]{DKOSZ}, the basis has Property (A, $1$). Now we use Proposition \ref{ps} to claim that the basis has the uniform Property (A). 

Finally, iii) $\Longrightarrow$ ii) is clear because Property (A) implies superdemocracy, see for example \cite[Proposition 3.8]{BDKOW2019}. 
\end{proof}

In the case of Banach spaces, $1$-greedy, $1$-almost greedy, and $1$-quasi-greedy bases have been studied in the literature \cite{AA2016, AA2017, AW2006}. Our next task is to find a similar characterization for $1$-consecutive greedy bases. First, we need an equivalence of a basis being Schauder.

\begin{defi}\normalfont
A basis $(\ee_n)_n$ is said to be consecutive unconditional if for some constant $\mathbf C\geqslant 1$,
$$\|x-P_I(x)\|\ \leqslant\ \mathbf C\|x\|, \forall x\in\mathbb{X}, \forall I\in \mathcal{I}.$$ In this case, we say that $(\ee_n)_n$ is $\mathbf C$-suppression consecutive unconditional. 
\end{defi}

\begin{rek}\normalfont
It is easy to see that a basis is consecutive unconditional if and only if it is Schauder. Also, a $1$-suppression consecutive unconditional basis is bimonotone: indeed, given  $x\in \XX$ with finite support and $n, m\in \NN$, 
\begin{align*}
\|S_{n+m}(x)-S_n(x)\|\ =\ \|S_{n+m}(x)-S_n(S_{n+m}(x))\|&\ \leqslant\ \|S_{n+m}(x)\|\\
&\ =\ \|x-P_{I}(x)\|\ \leqslant\ \|x\|,
\end{align*}
where $I=\{n+m+1, \dots, \max(\supp(x))\}$ if $S_{n+m}(x)\neq x$. On the other hand, a bimonotone basis is not necessarily $1$-suppression consecutive unconditional. Consider the following example. Let $(\ee_n)_n$ be the canonical basis of the space $\mathbb{X}$, which is the completion of $c_{00}$ under the following norm: for $x = (x_1, x_2, \ldots)$, 
$$\|x\|\ =\sup_{M\geqslant N\geqslant 1}\left|\sum_{n=N}^M x_n\right|.$$
Then the basis is bimonotone, but it is not $1$-suppression consecutive unconditional. To see this, just note that $\|(3, -1, 3, 0, 0, \ldots)\|=5$ whereas $\|(3, 0, 3, 0, 0, \ldots)\|=6$.  
\end{rek}

Our next result  is a counterpart for consecutive greedy bases of the result obtained in \cite[Theorem 3.4]{AW2006} for greedy ones. It gives a characterization of $1$-consecutive greedy bases in Banach spaces using the $1$-Property (A) and replacing $1$-suppression unconditionality with $1$-suppression consecutive unconditionality. 

\begin{thm}\label{m3}
Let $\BB$ be a basis for a Banach space $\XX$. The following are equivalent
\begin{enumerate}
\item[i)] $\BB$ is $1$-consecutive greedy. 
\item[ii)] $\BB$ is $1$-CGPCC. 
\item[iii)] $\BB$ has the $1$-Property (A) and is $1$-suppression consecutive unconditional. 
\end{enumerate}
\end{thm}

To prove Theorem \ref{m3}, we first need an useful lemma. 

\begin{lem}\label{l1}
If a basis $(\ee_n)_{n\in\mathbb N}$ has $\mathbf C$-Property (A), then
\begin{equation}\label{e5}\|x\|\ \leqslant\ A_p \mathbf C\|x - P_A(x) + \1_{\varepsilon B}\|,\end{equation}
for all $x\in\mathbb{X}$ with $\|x\|_\infty\leqslant 1$, for all finite sets $A, B\subset\mathbb{N}$ with $|A|\leqslant |B|$ and $(\supp(x-P_A(x))\cup A)\sqcup B$, and for all $\varepsilon, \delta\in \E$. 
\end{lem}

\begin{proof}
Let $x, A, B, \varepsilon$ be chosen as in \eqref{e5}. 
By \cite[Corollary 2.3]{AABW2021} and $\mathbf C$-Property (A), we have
\begin{align*}
    \|x\|\ =\ \left\|x- P_A(x) + \sum_{n\in A}\ee_n^*(x)\ee_n\right\|&\ \leqslant\ A_p \sup_{\delta}\|x-P_A(x) + \1_{\delta A}\|\\
    &\ \leqslant\ A_p\mathbf C\|x-P_A(x) + \1_{\varepsilon, B}\|. 
\end{align*}
This completes our proof. 
\end{proof}

\begin{prop}\cite[Proposition 2.5]{AA2017}\label{p1}
Let $(\ee_n)_{n}$ be a basis for a Banach space $\XX$. If $(\ee_n)_n$ has $1$-Property (A), then the basis is $1$-suppression quasi-greedy. 
\end{prop}

\begin{proof}[Proof of Theorem \ref{m3}]
That i) $\Longrightarrow$ ii) follows directly from definitions. 

Let us show ii) $\Longrightarrow$ iii).
Fix $x\in \mathbb{X}$ with $\|x\|_{\infty}\leqslant 1$ and choose distinct natural numbers $i,j \not\in \supp(x)$ and signs $\varepsilon_j,\varepsilon_i$.  We have 
$$
\|x+\varepsilon_i \ee_i\|\ =\ \|(x+\varepsilon_i \ee_i+\varepsilon_j \ee_j) -\varepsilon_j \ee_j\|\ \leqslant\ \|(x+\varepsilon_i \ee_i+\varepsilon_j \ee_j) -\varepsilon_i \ee_i\|\ =\ \|x+\varepsilon_j \ee_j\|.
$$
By induction, $(\ee_n)_n$ has $1$-Property (A).
To see that the basis is $1$-suppression consecutive unconditional, let $x\in \mathbb{X}$ and $I\in\mathcal{I}$. Choose $\alpha > 0$ sufficiently large such that if $y = \sum_{n\in I} (\alpha + \ee_n^*(x))\ee_n + \sum_{n\notin I}\ee_n^*(x)\ee_n$, then $I$ is a greedy sum of $y$. We have
$$\|x-P_I(x)\|\ =\ \|y-P_I(y)\|\ \leqslant\ \mathcal{D}^{con}_{|I|}(y)\ \leqslant\ \|y-\alpha \1_{I}\|\ =\ \|x\|.$$
Hence, $(\ee_n)_n$ is $1$-suppression consecutive unconditional. 

Finally, we prove that iii) $\Longrightarrow$ i).
Let $x\in\mathbb{X}$, $m\in\mathbb{N}$, $\Lambda\in G(x, m,1)$, and $I\in \mathcal{I}^{(m)}$. Choose $(a_n)_{n\in I}\subset \mathbb{F}$ arbitrarily. We need to show that
$$\|x - P_{\Lambda}(x)\|\ \leqslant\ \left\|x- \sum_{n\in I}a_n\ee_n\right\|.$$
Let $\alpha:= \min_{n\in \Lambda}|\ee_n^*(x)|$ and $\varepsilon = (\sgn(\ee_n^*(x)))$. Then $\|x-P_{\Lambda}(x)\|_\infty\leqslant \alpha$, so we have
\begin{align*}
    &\|x-P_{\Lambda}(x)\|\\
    &\ \leqslant\ \|x-P_{\Lambda}(x) - P_{I\backslash \Lambda}(x) + \alpha \1_{\varepsilon (\Lambda\backslash I)}\|\mbox{ by Lemma \ref{l1}}\\
&\ =\ \|x-P_{\Lambda\cup I}(x) + \alpha \1_{\varepsilon (\Lambda\backslash I)}\|\\
    &\ \leqslant\ \|x-P_{\Lambda\cup I}(x) + P_{\Lambda\backslash I}(x)\|\mbox{ by Proposition \ref{p1} and \cite[Lemma 2.5]{BBG2017}}\\
    &\ =\ \|x-P_{I}(x)\|\ \leqslant\ \left\|x- \sum_{n\in I}a_n\ee_n\right\|\mbox{ by $1$-consecutive unconditionality.}
\end{align*}
This completes our proof. 
\end{proof}

\section{More on Property (A, $\tau$)}\label{Atau}

In Theorem~\ref{m3}, the equivalence involves $1$-Property (A), but we know by Corollary~\ref{c2} that the basis also has Property $(A,\tau)$ for all $0<\tau<1$. So, it is natural to ask whether the constant is still $1$ for all such $\tau$. Our next result answers that question in the affirmative: in fact, we do not need to assume $1$-consecutive greediness to derive $1$-Property $(A,\tau)$ from $1$-Property $(A)$.

\begin{prop}\label{prop: 1enhancedpropertyA}Let $\BB$ be a basis for a Banach space $\XX$. If $\BB$ has $1$-Property $(A)$, then 
\begin{equation}
\|u+y\|\ \leqslant\ \|a u+z\|\label{lemma: enhancedpropertyAa}
\end{equation}
for every $a\in \FF$ with $|a|\geqslant 1$ and every $u,y,z\in \XX$ satisfying:
\begin{itemize}
\item $\supp(u)\cap (\supp(y)\cup \supp(z))=\emptyset$.
\item $|\supp(y)|\leqslant |\supp(z)|<\infty$. 
\item $\|u+y\|_{\infty}\leqslant 1$. 
\item $\min_{n\in \supp(z)}|\ee^*_n(z)|\geqslant 1$. 
\end{itemize}
In particular,
\begin{equation}
\|\tau x+\1_{\varepsilon,A}\|\ \leqslant\ \|x+\1_{\delta ,B}\|\label{lemma: enhancedpropertyAb}
\end{equation}
for all $0<\tau\leqslant 1$, for all $x\in \XX$ with $\|x\|_{\infty}\leqslant \tau^{-1}$, for all finite sets $A,B\subset \mathbb{N}$ with $|A|\leqslant |B|$ and $\supp(x)\cap (A\cup B)=\emptyset$, and for all $\varepsilon, \delta\in \E$.
\end{prop}
\begin{proof}
First, we prove  \eqref{lemma: enhancedpropertyAa}: choose $\theta\in\mathbb{F}$ with $|\theta|=1$ so that $|a|=a\theta$  and set $z_1:=\theta z$, $A:=\supp(y)$, and $B:=\supp(z_1)$. By convexity, there is $\varepsilon\in \E$ such that 
\begin{equation}
\|u+y\|\ \leqslant\ \|u+\1_{\varepsilon, A}\|.  \label{lemma: enhancedpropertyA step1 }
\end{equation}
Set $\delta_n:=\sgn(\ee_n^*(z_1))$ for all $n\in B$. By $1$-Property (A),
\begin{equation}
\|u+\1_{\varepsilon, A}\|\ \leqslant\ \|u+\1_{\delta, B}\|. \label{lemma: enhancedpropertyA step2 }
\end{equation}
By the Hahn-Banach theorem, there exists $x^{*}\in \XX^*$ with $\|x^*\|=1$ so that 
$$
x^*(u+\1_{\delta, B})\ =\ \|u+\1_{\delta, B}\|.
$$
For each $n\in B$, pick $\delta_n'$ with $|\delta'_n|=1$ so that $\delta'_n\delta_n x^*(\ee_n)=|x^*(\ee_n)|$, and choose $\gamma\in \FF$ with $|\gamma|=1$ so that $\gamma x^*(u)=|x^*(u)|$. We have 
\begin{align*}
x^*(u)+\sum_{n\in B}\delta_n x^*(\ee_n)\ =\ \|u+\1_{\delta, B}\|\ =\ \left\| \gamma u+ \1_{\delta'\delta, B}\right\|&\ \geqslant\ \left| x^*(\gamma u+ \1_{\delta'\delta, B})\right|\\
&\ =\ |x^*(u)|+\sum_{n\in B}| x^*(\ee_n)|,
\end{align*}
where the second equality is due to $1$-Property (A). 
Therefore, $x^*(u)=|x^*(u)|$ and $\delta_n x^*(\ee_n)=|x^*(\ee_n)|$ for all $n\in B$. Hence, 
\begin{align*}
\left\| a u+ z \right\|&\ =\ \left\| |a| u+ z_1\right\|\ \geqslant\ |x^{*}(|a|u+z_1)|\\
&\ =\ \left|  |a| x^{*}(u)+\sum_{n\in B}\ee_n^*(z_1)   x^{*}(\ee_n)\right|\ =\ |a| |x^{*}(u)|+\sum_{n\in B}|\ee_n^*(z_1)||x^{*}(\ee_n)|\\
&\ \geqslant\ |x^{*}(u)|+\sum_{n\in B} |x^{*}(\ee_n)|\ =\ \|u+\1_{\delta, B}\|. 
\end{align*} 
Combining the above with \eqref{lemma: enhancedpropertyA step1 } and \eqref{lemma: enhancedpropertyA step2 }, we obtain \eqref{lemma: enhancedpropertyAa}. 

Finally, \eqref{lemma: enhancedpropertyAb} follows from \eqref{lemma: enhancedpropertyAa} by taking $u=\tau x$, $a=\tau^{-1}$, $y=\1_{\varepsilon,A}$ and $z=\1_{\delta, B}$.
\end{proof}

Next, we answer \cite[Problem 7.3]{C3}, which asked whether Property (A, $\tau$) for some $\tau\in (0,1]$ implies Property (A, $t$) for all $t\in (0,1]$ (with possibly different constants). The proof leverages a recent result in \cite{AAB2022}, which showed the surprising equivalence between the so-called quasi-greedy for largest coefficients and near unconditionality. We recall these definitions and a main result in \cite{AAB2022}.

\begin{defi}\normalfont\label{defQGLC}\cite[Definition 4.6]{AABW2021}
A basis $\mathcal{B}$ in a $p$-Banach space $\mathbb X$ is quasi-greedy for largest coefficients (QGLC) if there exists $\mathbf C\geqslant 1$ such that
$$\|1_{\varepsilon A}\|\ \leqslant\ \mathbf C\|x + 1_{\varepsilon A}\|,$$
for any finite $A\subset\mathbb{N}$, sign $\varepsilon$, and $x\in\mathbb{X}$ with $\|x\|_\infty\leqslant 1$ and $\supp(x)\sqcup A$. 
\end{defi}

\begin{defi}\normalfont\label{nearUnc}\cite{E1978}
A basis $\mathcal{B}$ in a $p$-Banach space $\mathbb X$ is nearly unconditional if for every $t\in (0,1]$, there is a constant $\phi(t)\geqslant 1$ such that
$$\|P_A(x)\|\ \leqslant\ \phi(t)\|x\|,$$
for all $x\in\XX$ with $\|x\|_\infty\leqslant 1$ and for $A\in\mathbb{N}^{<\infty}$ satisfying $\min_{n\in A}|\ee^*_n(x)|\geqslant t$. 
\end{defi}

\begin{thm}\cite[Theorem 2.6]{AAB2022}\label{QGLC+nearlyunc}
A basis $\mathcal{B}$ in a $p$-Banach space $\mathbb X$ is QGLC if and only if it is nearly unconditional. 
\end{thm}

We are ready to prove the second main result in this section. 

\begin{prop}\label{lemmaalltau}Let $0<\tau, p\leqslant 1$  and $\BB$ be a basis of a $p$-Banach space $\XX$. If $\BB$ has $\mathbf C$-Property $(A,\tau)$, then for every $0<t\leqslant 1$, there is $\mathbf D(t)\geqslant 1$ such that $\BB$ has $\mathbf D(t)$-Property $(A,t)$. 
\end{prop}
\begin{proof}
By \cite[Proposition 5.3]{C3}, $\BB$ has $(\mathbf C/t)$-Property (A). Thus, by \cite[Proposition 5.3]{AABW2021}, it is quasi-greedy for largest coefficients. By Theorem \ref{QGLC+nearlyunc}, $\BB$ is nearly unconditional. Let $\phi: (0,1]\rightarrow [1,\infty)$ be its near unconditionality function, and fix $A, B, \varepsilon, \delta$, and $x$ as in the definition of Property $(A,t)$. We have 
\begin{align*}
\|t x\|^p\ \leqslant\  \|tx +t \1_{\delta, B}\|^p+\|t\1_{\delta, B}\|^p&\ \leqslant\ \|tx +t \1_{\delta, B}\|^p+\phi^p(t)\|t\1_{\delta, B}+t x\|^p\\
&\ =\ t^p(1+\phi^p(t))\|x+\1_{\delta,B}\|^p. 
\end{align*}
On the other hand, 
$$
\|\1_{\varepsilon, A}\|\ \leqslant\ \frac{\mathbf C}{\tau t}\|t \1_{\delta, B}\|\ \leqslant\ \frac{\mathbf C}{\tau t}\phi(t)\|t\1_{\delta, B}+t x\|=\frac{\mathbf C}{\tau}\phi(t)\|x+\1_{\delta, B}\|. 
$$
Therefore, $p$-convexity gives
$$
\|t x +\1_{\varepsilon,A}\|\ \leqslant\ \left(t^p+t^p\phi^p(t)+\frac{\mathbf C^{p}}{\tau^p}\phi^p(t)\right)^{\frac{1}{p}}\|x+\1_{\delta, B}\|,
$$
and the proof is complete.  
\end{proof}

We close this section with an observation on the uniform Property (A). We have already considered the case where the uniform bound is $1$ and the space is locally convex, but not the general case. It turns out that the uniform property (A) is equivalent to the property obtained from Property (A) by removing the restriction on the $\ell_{\infty}$-norm of $x$. More precisely, we have the following result. 
\begin{lem}\label{lemmacaracunif}Let $\BB$ be a basis of a $p$-Banach space $\XX$. The following are equivalent. 
\begin{enumerate}
\item [i)] $\BB$ has the uniform property $(A)$. 
\item[ii)] There is a constant $\mathbf C\geqslant 1$ such that
$$\|x+ \1_{\varepsilon,  A}\|\ \leqslant\ \mathbf C\|x+ \1_{\delta, B}\|,$$
for all $x\in\mathbb{X}$ and all finite sets $A, B\subset\mathbb{N}$ with $|A|\leqslant |B|$ and $\supp(x)\sqcup A\sqcup B$, and for all signs $\varepsilon, \delta$. 
\end{enumerate}
\end{lem}
\begin{proof}
i) $\Longrightarrow$ ii) Let $\CC\geqslant 1$ be a uniform upper bound for the Property (A,$\tau$) constants. Assume $\|x\|_{\infty}>1$ and pick $\tau < 1/\|x\|_{\infty}$. Then 
\begin{align*}
\|\tau x+ \1_{\varepsilon,  A}\|\ \leqslant\ \mathbf C\|x+ \1_{\delta, B}\|.
\end{align*}
Letting $\tau$ tend to zero, we get $\|\1_{\varepsilon,  A}\| \leqslant \mathbf C\|x+ \1_{\delta, B}\|$. Hence, 
\begin{align*}
\| x+ \1_{\varepsilon,  A}\|^p&\ \leqslant\  \| x+ \1_{\delta,  B}\|^p+\|\1_{\varepsilon,  A}\|^p+\|\1_{\delta,  B}\|^p\\
&\ \leqslant\ \| x+ \1_{\delta,  B}\|^p+\|\1_{\varepsilon,  A}\|^p+\CC^p\|\1_{\varepsilon,  A}\|^p\\
&\ \leqslant\ (1+\CC^p+\CC^{2p})\|x+ \1_{\delta, B}\|^p.
\end{align*}
ii) $\Longrightarrow$ i). Fix $x, A, B, \varepsilon, \delta$ as in Definition~\ref{defatau}. If $\|\tau x\|\geqslant 2^{\frac{1}{p}}\CC \|\1_{\varepsilon, A}\|$, then 
$$\|\1_{\delta,B}\|^p\ \leqslant\ \CC^p \|\1_{\varepsilon, A}\|^p\ \leqslant\ \frac{\|\tau x\|^p}{2},$$ so
$$
\|x+\1_{\delta, B}\|^p\ \geqslant\  \|x\|^p-\|\1_{\delta,B}\|^p\ \geqslant\ \frac{\|x\|^p}{2}.
$$
Hence, we have
\begin{align*}
\|\tau x+\1_{\varepsilon,  A}\|^p&\ \leqslant\  \|\1_{\varepsilon,  A}\|^p+\|\tau x\|^p\\
&\ \leqslant\ \tau^p\left(1+\frac{1}{2\CC^p}\right)\|x\|^p\ \leqslant\  2\tau^p\left(1+\frac{1}{2\CC^p}\right)\|x+\1_{\delta,B}\|^p.
\end{align*}
On the other hand, if $\|\tau x\|\leqslant 2^{\frac{1}{p}}\CC \|\1_{\varepsilon, A}\|$, then 
\begin{align*}
\|\tau x+\1_{\varepsilon,  A}\|^p&\ \leqslant\ \|\1_{\varepsilon,  A}\|^p+\|\tau x\|^p \leqslant \left(2\CC^p+1\right)\|\1_{\varepsilon,  A}\|^p\\
&\ \leqslant\ \left(2\CC^p+1\right)2^{1-p}\max\{\|x+\1_{\varepsilon,  A}\|^p,\|x-\1_{\varepsilon,  A}\|^p\}\\
&\ \leqslant\ \left(2\CC^p+1\right)2^{1-p}\CC^p\|x+\1_{\delta,B}\|^p.
\end{align*}
Comparing the upper bounds and combining, we get 
$$
\|\tau x+\1_{\varepsilon,  A}\|\ \leqslant\ \left(2\CC^p+1\right)^{\frac{1}{p}}2^{\frac{1}{p}-1}\CC\|x+\1_{\delta,B}\|.
$$

\end{proof}

\section{On squeeze symmetric bases}\label{sectionsqueezesymmetric}

In Definition~\ref{definitionpoly}, we restricted the approximants on the right-hand side of the inequality to vectors with constant coefficients with respect to a basis, and Theorem~\ref{m1} shows that the resulting definition is equivalent to Definition~\ref{definitionconsecutivegreedy}. This sort of approximation was studied previously in \cite{BBC, BB2017}, which give characterizations of familiar greedy-like properties. In light of those results, we ask what happens if we restrict our approximations to such vectors on both sides of our inequalities.  Do we still get a familiar greedy-like property? 
More precisely, we study some conditions under which there is a constant $\CC>0$ such that 
\begin{equation}\label{ef1}
\|x-\lambda \1_{A}\|\ \leqslant\ \CC\|x-\lambda \1_{B}\|
\end{equation}
for each $x\in \XX$, $A\in G(x)$ and certain sets $B$ (to be specified latter).
It turns out that, for Schauder bases (or Markushevich bases in locally convex spaces), we can use \eqref{ef1} to characterize squeeze symmetric bases, which are sandwiched between symmetric bases with equivalent fundamental functions (see \cite{AABBL2021, AAB2021, AABW2021}). For our purposes, it is convenient to use the following equivalence, which follows from  \cite[Proposition 9.4, Theorem 9.14]{AABW2021}.
\begin{lem}\label{lemma: sqscar}A basis $\BB$ for a $p$-Banach space $\XX$ is squeeze symmetric if and only if there is a $\CC>0$ such that 
\begin{equation}
\|\1_{\varepsilon, A}\|\ \leqslant\ \CC\|x\|\label{lemma: sqscar2}
\end{equation}
for all finite $A\subset \NN$ finite, all $\varepsilon\in \E$, and all $x\in \XX$ such that 
$$\left|\{n\in \NN: |\ee_n^*(x)|\geqslant 1\}\right|\ \geqslant\ |A|.$$ 
The smallest $\CC$ for which \eqref{lemma: sqscar2} holds is denoted $\CC_{sqs}$. 
\end{lem}
To prove our result, we need the following lemma, which follows from \cite[Proposition 3.11]{BL2021}.  

\begin{lem}\label{lemma separation}Let $\BB$ be a Markushevich basis of a Banach space $\XX$. There is an $\M>0$ such that, for every finite set $F\subset \NN$ and each $m\in\NN$, there is $E\subset \NN$ such that $E>F$, $|E|=m$, and
\begin{equation}
\|x\|\ \leqslant\ \M\|x+y\|, \forall x\in [\ee_n: n\in F], \forall y\in [\ee_n: n\in E].\label{lemma separation: sep}
\end{equation}
\end{lem}

Now we can give a characterization of squeeze symmetric Schauder bases in terms of approximations by one-dimensional subspaces. Also, for $p=1$, we give a characterization valid for Markushevich bases as well.

\begin{prop}\label{proposition: squeeze symmetric}Let $\BB$ be a basis of a $p$-Banach space $\XX$. Consider the following statements.
\begin{enumerate}
\item[i)] $\BB$ is squeeze symmetric. 
\item[ii)] There is $\CC>0$ such that 
\begin{align*}
&\|x+\lambda \1_{\varepsilon, A}\|\ \leqslant\  \CC\|x+\lambda \1_{\delta, B}\|,\\
& \forall x\in \XX, A\in G(x), B\subset \NN : |B|\geqslant |A|, |B\cap \supp(x)|\leqslant |A|,\\
&\forall\varepsilon, \delta\in \E:\varepsilon_n=\delta_n, \forall n\in A\cap B,\\
&\forall \lambda \in \FF. 
\end{align*}
\item[iii)] There is $\CC>0$ such that, 
\begin{align}
&\|x- \1_{A}\|\ \leqslant\ \CC\|x-\1_{\varepsilon, B}\|\nonumber\\
&\forall x\in \XX, A\in G(x), B\subset \supp(x), |B|= |A|,  A\cap B=\emptyset, \varepsilon\in \E. \nonumber
\end{align}
\end{enumerate}
The following implications hold: i) $\Longrightarrow$ ii) $\Longrightarrow$ iii). Moreover, if $p=1$ and $\BB$ is a Markushevich basis, or $0<p\leqslant 1$ and $\BB$ is a Schauder basis, then  iii) $\Longrightarrow$  i).
\end{prop}
\begin{proof}
i) $\Longrightarrow$ ii): set $D:=A\cap B$, $A_1:=A\setminus B$, $B_1:=B\setminus A$, $y:=x+\lambda \1_{\varepsilon, D}$, and $a:=\min_{n\in A}|\ee_n^*(x)|$. We shall show that
\begin{equation}
\|y+\lambda \1_{\varepsilon, A_1}\|\ \leqslant\  (1+2^{p+1}\CC^p_{sqs})^{\frac{1}{p}}\|y+\lambda \1_{\delta, B_1}\|. \label{proposition: squeeze symmetric: to prove (1)}
\end{equation}
To that end, we assume $\lambda\neq0$, and consider two following cases.

Case 1: if $|\lambda|>2a$, then $$|\ee_n^*(y+\lambda \1_{\delta, B_1})|\ \geqslant\ |\lambda|-a\ >\  \frac{|\lambda|}{2}, \forall n\in B_1.$$ 
Hence, 
$$
\left| \left\lbrace n\in \NN: |\ee_n^*(y+\lambda \1_{\delta, B_1})|\geqslant |\lambda|/2\right\rbrace\right| \ \geqslant\ |B_1|\ \geqslant\ |A_1|.$$
It follows that 
\begin{align*}
\|y+\lambda \1_{\varepsilon, A_1}\|^p&\ \leqslant\ \|y+\lambda \1_{\delta, B_1}\|^p+2^p \|2^{-1}\lambda \1_{\varepsilon, A_1}\|^p+2^p\|2^{-1}\lambda \1_{\delta, B_1}\|^p\\
&\ \leqslant\ (1+2^{p+1}\CC^p_{sqs})\|y+\lambda \1_{\delta, B_1}\|^p\mbox{ by squeeze symmetry}.
\end{align*}

Case 2:  if $0<|\lambda|\leqslant 2a$, then  $$|\ee_n^*(y+\lambda \1_{\delta, B_1})|\ \geqslant\ |\lambda|/2, \forall n\in A_1\sqcup (B_1\setminus \supp(x)).$$ Also, since $|\lambda| > 0$, we have $A\subset \supp(x)$, which,  combined with the hypothesis $|\supp(x)\cap B|\leqslant |A|$, implies that $|B_1\cap \supp(x)|\leqslant |A_1|$. Hence, 
\begin{align*}
\left| \left\lbrace n\in \NN: |\ee_n^*(y+\lambda \1_{\delta, B_1})|\ \geqslant\ |\lambda|/2\right\rbrace\right|&\ \geqslant\
|A_1\sqcup (B_1\setminus \supp(x))|\\
&\ =\ |A_1|+|B_1\setminus \supp(x)|\ \geqslant\ |B_1|\ \geqslant\ |A_1|.
\end{align*}
As in Case 1, we obtain 
$$
\|y+\lambda \1_{\varepsilon, A_1}\|^p\ \leqslant\ (1+2^{p+1}\CC^p_{sqs})\|y+\lambda \1_{\delta, B_1}\|^p.
$$
We have proved iii).

ii) $\Longrightarrow$ iii) is immediate. 

iii) $\Longrightarrow$ i): we borrow the argument from the proofs of \cite[Theorem 1.10]{B2019} and \cite[Theorem 4.2]{BL2021}. Let $\M$ be the basis constant if $\BB$ is a Schauder basis or the constant in Lemma~\ref{lemma separation} if $p=1$ and $\BB$ is a Markushevich basis. Fix $x\in \XX$, $A\neq\emptyset$, and $\varepsilon$ as in Lemma~\ref{lemma: sqscar}.  By density and a standard perturbation argument, we may assume that $x$ has a finite support. Choose $E>\supp(x)\cup A$ with $|E|=|A|$ so that \eqref{lemma separation: sep} holds. Given that $E\in G(\1_{\varepsilon,A}+\1_{E},|A|,1)$, we have 
\begin{equation}\label{e101}
\|\1_{\varepsilon,A}\|\ =\ \|(\1_{\varepsilon,A}+\1_{E})-\1_{E} \|\ \leqslant\ \CC \|(\1_{\varepsilon,A}+\1_{E})-\1_{\varepsilon,A}\|\ =\ \CC\|\1_{E}\|.
\end{equation}
Now pick $B\in G(x, |E|, 1)$. Since $B\in G(x+\1_{E}, |E|, 1)$ and $B\cap E=\emptyset$, 
\begin{align}\label{e102}
\|\1_{E}\|^p&\ \leqslant\ \|x-\1_{B}+\1_{E}\|^p+\|x-\1_{B}\|^p\ \leqslant\ (1+\M^p)\|x+\1_{E}-\1_{B}\|^p\nonumber\\
&\ \leqslant\ \CC^p(1+\M^p)\|x+\1_{E}-\1_{E}\|^p\ =\ \CC^p(1+\M^p)\|x\|^p. 
\end{align} 
Combining \eqref{e101} and \eqref{e102} gives
$$
\|\1_{\varepsilon,A}\|\ \leqslant\ \CC^2(1+\M^p)^{\frac{1}{p}}\|x\|. 
$$
This completes our proof.
\end{proof}

\begin{cor}\label{corollaryforsemi}If $\BB$ is a squeeze symmetric basis, there is $\CC>0$ such that 
\begin{equation}
\min_{\substack{B\subset A\\\lambda_1\in \FF}}\|x-\lambda_1 \1_{B}\|\ \leqslant\ \CC\inf_{\substack{D\in \NN^{<\infty}\\|D\cap \supp(x)|\leqslant |A|\\\lambda_2\in \FF }}\|x-\lambda_2\1_{D}\|, \forall x\in \XX,  \forall A\in G(x). \label{corollaryforsemi:min}
\end{equation}
\end{cor}
\begin{proof}
Let $\CC$ be the constant of Proposition \ref{proposition: squeeze symmetric} item ii). Fix $x\in \XX$, $A\in G(x)$, $\lambda_2 \in \FF$ and a nonempty finite set $D\subset \NN$ such that $|D\cap \supp(x)|\leqslant |A|$.  

If $|D|\geqslant |A|$, then $\|x-\lambda_2\1_{A}\|\leqslant \CC\|x-\lambda_2\1_{D}\|$. 

If $|D|<|A|$, choosing $A_1\subset A$ so that $A_1\in G(x,|D|,1)$, we get $\|x-\lambda_2\1_{A_1}\|\leqslant \CC\|x-\lambda_2\1_{D}\|$. 

Since $A$ is a finite set, the minimum in \eqref{corollaryforsemi:min} is achieved, and the proof is complete. 
\end{proof}
\section{Appendix: approximation by pseudo-greedy sets}

So far, we have consider ``consecutive'' approximations in which we use approximations by projections on intervals (\cite{BBC}) or vectors supported on intervals (Definition~\ref{definitionconsecutivegreedy}.) One might wonder what happens if, instead, we replace greedy sets for some kind of ``consecutive'' variant. In this appendix, we consider a natural weakening of the greedy condition on sets.  
\begin{defi}\label{definition pseudo greedy}Let $\BB$ be a basis of a $p$-Banach space $\XX$ and $x\in \XX$. A set $A\subset \NN$ is a \rm{pseudo-greedy} set of $x$ if, for every $n\in \NN\setminus A$, either 
\begin{align*}
|\ee_n^*(x)|\ \geqslant\  \sup_{k\in A}|\ee_k^*(x)|,
\end{align*}
or 
\begin{align*}
|\ee_n^*(x)|\ \leqslant\ \inf_{k\in A}|\ee_k^*(x)|.
\end{align*}
\end{defi}
It is not hard to show that a set is pseudo-greedy if and only if it is the difference of two greedy sets.
\begin{lem}\label{lemmapseudogreedy}Let $\BB$ be a basis of a $p$-Banach space $\XX$ and $\CC>0$. The following hold
\begin{enumerate}
\item[i)] Suppose that for every $x\in \XX$ and $m\in \NN$, there is a pseudo-greedy set $A$ of $x$ such that $|A| = m$ and
\begin{align*}
\|x-P_A(x)\|\ \leqslant\ \CC\sigma_m(x). 
\end{align*}
Then $\BB$ is $\CC$-greedy. 
\item[ii)] Suppose that for every $x\in \XX$ and $m\in \NN$, there is a pseudo-greedy set $A$ of $x$ such that $|A| = m$ and 
\begin{align*}
\|x-P_A(x)\|\ \leqslant\ \CC\widetilde{\sigma}_m(x). 
\end{align*}
Then $\BB$ is $\CC$-almost greedy.
\item[iii)] Suppose that for every $x\in \XX$ and $m\in \NN$, there is a pseudo-greedy set $A$ of $x$ such that $|A|=m$ and 
\begin{align*}
\|x-P_A(x)\|\ \leqslant\ \CC\sigma_m^{con}(x). 
\end{align*}
Then $\BB$ is $(\CC^p+c_2^{2p})^{\frac{1}{p}}$-consecutive greedy. 
\end{enumerate}
\end{lem}
\begin{proof}
The proofs of i) and ii) are essentially the same. We prove the latter. Fix $x\in \XX$. By a density argument, we may suppose that $x\in [\ee_n: n\in \NN]$ and that $|\ee_m^*(x)|\neq |\ee_n^*(x)|$ for all $m,n\in \supp(x), n\neq m$. Fix $m\in \NN$ and $A\in G(x,m,1)$. We may further assume that $x\neq P_A(x)\neq 0$. Choose $n_0\not \in \supp(x)$ and for each $k\in \NN$, let $y_k:=k \ee_{n_0}+x$. By hypothesis, for each $k\in \NN$, there is a pseudo-greedy set $A_k$ of $y_k$ such that $|A_k|=m+1$ and 
$$
\|y_k-P_{A_k}(y_k)\|\ \leqslant \ \CC \widetilde{\sigma}_{m+1}(y_k)\ \leqslant\ \CC\widetilde{\sigma}_m(x)\ \leqslant\  \CC\|x\|. 
$$
Hence, the left-hand side of the above inequality is bounded as $k$ tends to infinity, so there is $k_0\in \NN$ such that $n_0\in A_{k_0}$ and $k_0>\|x\|_{\infty}$. By the definition of a pseudo-greedy set, $|e_n^*(y_k)|\leqslant \inf_{j\in A_{k_0}}|e_j^*(x)|$ for all $n\notin A_{k_0}$. Hence, $A_{k_0}\in G(y_{k_0}, m+1, 1)$. Since $G(x,m,1)=\{A\}$, we have $G(y_{k_0}, m+1, 1)=\{A\cup \{n_0\}\}$, which implies that $\|x-P_A(x)\|=\|y_k-P_{A_k}(y_k)\|$. This completes our proof.

We prove iii). Fix $x\in \XX$, $m\in \NN$ and $A\in G(x,m,1)$. As before, we assume that
$$
x\ \in\ [\ee_n: n\in \NN];\quad |\ee_i^*(x)|\ \neq\ |\ee_j^*(x)|\;\forall i\neq j; \quad x\ \neq\ P_A(x) \ \neq\ 0.
$$
Given $I\in \II^{(m)}$ and $y\in [\ee_n: n\in I]$, if $I\not=A$ choose $n_0\in I\setminus A$ and $n_1\in A\setminus I$. For each $k$, let $y_k:=x+k \ee_{n_0}$. As in the proof of ii), for sufficiently large $k$, the only pseudo-greedy set of $y_k$ of cardinality $m+1$ that meets the condition of the statement is $A_0:=A\cup \{n_0\}\in G(y_k,m+1,1)$. Hence, 
\begin{align*}
\|x-P_A(x)\|^p&\ =\ \|y_k-P_{A_0}(y_k)+\ee_{n_0}^*(x)\ee_{n_0}\|^p\\
&\ \leqslant\ \left(\CC\sigma_{m+1}^{con}(y_k)\right)^{p}+|\ee_{n_0}^*(x)|^p\|\ee_{n_0}\|^p\\
&\ \leqslant\ \CC^p\|x-y\|^p+|\ee_{n_1}^*(x-y)|^p\|\ee_{n_0}\|^p\\
&\ \leqslant\ (\CC^p+c_2^{2p})\|x-y\|^p. 
\end{align*}
On the other hand, if $I=A$, let $n_0:=\max(A)+1$ and, for each $k$, let $y_k:=x+k\ee_{n_0}$. As before, for sufficiently large $k$, the only pseudo-greedy set of $y_k$ of cardinality $m+1$ that meets the condition of the statement is $A_0:=A\cup \{n_0\}\in G(y_k,m+1,1)$. Hence, 
\begin{align*}
\|x-P_A(x)\|^p&\ =\ \|y_k-P_{A_0}(y_k)+\ee_{n_0}^*(x)\ee_{n_0}\|^p\\
&\ \leqslant\ \left(\CC\sigma_{m+1}^{con}(y_k)\right)^{p}+|\ee_{n_0}^*(x-y)|^p\|\ee_{n_0}\|^p\\
&\ \leqslant\ (\CC^p+c_2^{2p})\|x-y\|^p. 
\end{align*}
This completes our proof.
\end{proof}

\bigskip
	\noindent \textbf{Funding:} The first author was partially supported by CONICET PIP 0483 and ANPCyT PICT-2018-04104, and has received funding from the European Union’s Horizon 2020 research and innovation programme under the Marie Skłodowska-Curie grant agreement No 777822. The second author was partially supported by the Grant PID2019-105599GB-I00 (Agencia Estatal de Investigación, Spain).

\end{document}